\pdfoutput=1
\documentclass[colorlinks = true, a4paper, oneside, reqno, 12pt]{amsart}

\usepackage[top = 3.5cm, bottom = 3.5cm, left = 2.5cm, right = 2.5cm]{geometry}
\usepackage[T1]{fontenc}
\usepackage{lmodern}

\usepackage[square,comma,numbers]{natbib}

\usepackage[shortlabels]{enumitem}

\usepackage{mathtools, amssymb}

\usepackage{tikz-cd}

\usepackage{doi}

\newtheorem{theorem}{Theorem}[section]
\newtheorem{lemma}[theorem]{Lemma}
\newtheorem{proposition}[theorem]{Proposition}
\newtheorem{corollary}[theorem]{Corollary}

\theoremstyle{definition}
\newtheorem{defn}[theorem]{Definition}
\newtheorem{example}[theorem]{Example}

\theoremstyle{remark}
\newtheorem{remark}[theorem]{Remark}
\newtheorem{notation}[theorem]{Notation}

\numberwithin{equation}{section}

\usepackage{hyperref}
\hypersetup{allcolors=[rgb]{0.1,0.1,0.4}}

\newcommand{\Set}{\mathbf{Set}}
\newcommand{\CAT}{\mathsf{Cat}}
\newcommand{\COF}{\mathsf{Cof}}
\newcommand{\SPAN}{\mathsf{Span}}
\newcommand{\LENS}{\mathsf{Lens}}
\newcommand{\SOPF}{\mathsf{SOpf}}
\newcommand{\E}{\mathcal{E}}

\newcommand{\phibar}{\overline{\varphi}}
\newcommand{\br}{\rightleftharpoons}

\makeatletter
\@namedef{subjclassname@2020}{%
  \textup{2020} Mathematics Subject Classification}
\makeatother

\newcommand{\etalchar}[1]{$^{#1}$}

\begin{document}

\title{Internal split opfibrations and cofunctors}

\author{Bryce Clarke}

\address{Centre of Australian Category Theory\\
Macquarie University, NSW 2109, Australia}

\email{bryce.clarke1@hdr.mq.edu.au}

\thanks{The author is supported by the 
Australian Government Research Training Program Scholarship.}



%\date{}

%\dedicatory{}

%----------------------------------------------------------------------%
% Abstract                                                             %
%----------------------------------------------------------------------%
\begin{abstract}
Split opfibrations are functors equipped with a suitable choice of 
opcartesian lifts. 
The purpose of this paper is to characterise internal split opfibrations
through separating the structure of a suitable choice of lifts from the 
property of these lifts being opcartesian. 
The underlying structure of an internal split opfibration is captured by
an internal functor equipped with an internal cofunctor, while the 
property may be expressed as a pullback condition, akin to the simple 
condition on an internal functor to be an internal discrete opfibration. 
Furthermore, this approach provides two additional characterisations
of internal split opfibrations, via the d\'{e}calage construction and 
strict factorisation systems. 
For small categories, this theory clarifies several aspects of delta 
lenses which arise in computer science. 
\end{abstract}

\maketitle

%----------------------------------------------------------------------%
% Section 1: Introduction                                              %
%----------------------------------------------------------------------%
\section{Introduction}
\label{sec:introduction}

Split opfibrations are functors equipped with a functorial choice of 
opcartesian lifts, and they have been widely studied since the seminal
work of \cite{Gra66}. 
In the paper \cite{Str74}, split opfibrations were generalised to an 
arbitrary $2$-category $\mathcal{K}$, and this definition readily 
specialises to the case $\mathcal{K} = \CAT(\E)$, for a fixed category 
$\E$ with pullbacks, to yield internal split opfibrations. 
However, there are other ways to define internal split opfibrations, 
including in \cite{Joh77} as internal categories in the category of 
internal diagrams, or in the recent paper \cite{vG18} using lax 
codescent objects.

This paper aims to develop a new approach to internal split 
opfibrations, using the theory of \emph{cofunctors} from \cite{Agu97}.  
The structure of an \emph{internal lens} 
(Definition~\ref{defn:internallens}) is introduced as an internal 
functor $f \colon A \rightarrow B$ equipped with an  internal cofunctor 
$\varphi \colon B \nrightarrow A$, and is equivalent to a commutative 
diagram of internal functors, 
\begin{equation}
\label{eqn:intro}
\tag{$\ast$}
\begin{tikzcd}[column sep = small, row sep = small]
& \Lambda
\arrow[ld, "\varphi"']
\arrow[rd, "\phibar"]
& \\
A 
\arrow[rr, "f"']
& & B
\end{tikzcd}
\end{equation}
where $\varphi$ is identity-on-objects and $\phibar$ is an internal 
discrete opfibration. 
An internal split opfibration is defined as an internal lens satisfying 
a certain property (Definition~\ref{defn:internalsplitopf}).
The main result (Theorem~\ref{thm:main}) will be to characterise 
internal split opfibrations via a natural condition on \eqref{eqn:intro} 
with respect to the right d\'{e}calage comonad on $\CAT(\E)$.
An internal lens with a certain strict factorisation system is also 
shown to characterise internal split opfibrations 
(Theorem~\ref{thm:johnstone}), providing a link to the explicit 
axiomatisation in \cite{Joh77}. 

The motivation for this paper arises from both theoretical and applied 
concerns. 
In part, there was a goal to find a defining property of internal 
split opfibrations which closely resembles the remarkably simple 
condition on an internal functor to be an internal discrete opfibration 
(Definition~\ref{defn:internaldopf}). 
However, split opfibrations are functors with 
\emph{additional structure},
and it becomes necessary to unravel this definition to discover the 
appropriate internal structure with the desired property. 
This key idea is realised by separating an internal split opfibration 
into three parts: an internal functor, an internal cofunctor
(Definition~\ref{defn:internalcofun}), 
and the defining property expressed as a pullback condition.

Internal cofunctors are a kind of morphism between internal categories.
They were developed in the thesis \cite[Chapter 4]{Agu97}, 
and generalise the earlier concept of comorphisms between vector bundles 
(and other related structures) introduced in \cite{HM93}.
In the ordinary case (Example~\ref{ex:cofunctor}), 
a cofunctor between categories is a functorial \emph{lifting} of 
morphisms in the opposite direction to an assignment on objects, thus 
providing an ideal structure to describe the \emph{splitting} of a split
opfibration. 
An important result in \cite{HM93} is that every cofunctor 
$\varphi \colon B \nrightarrow A$ may be represented as a span of 
internal functors, 
\begin{equation*}
\begin{tikzcd}[column sep = small, row sep = small]
& \Lambda
\arrow[rd, "\varphi"]
\arrow[ld, "\phibar"']
& \\
B 
& & A
\end{tikzcd}
\end{equation*}
where $\phibar$ is an internal discrete opfibration and $\varphi$ is 
identity-on-objects (Proposition~\ref{prop:cofunctor}).
This reveals how cofunctors may be naturally understood 
as morphisms between categories, and clarifies the abstract 
composition of cofunctors (Remark~\ref{rem:cofunctorcat}) 
through the concrete composition of spans. 

There is also another natural way of understanding cofunctors 
as morphisms between categories. 
Recall that for a category $\E$ with pullbacks, an internal category 
is a formal monad in the bicategory $\SPAN(\E)$ of spans. 
An internal functor corresponds to a lax monad morphism whose $1$-cell
component is a \emph{right adjoint} in $\SPAN(\E)$. 
Dually, an internal cofunctor corresponds to a lax monad morphism whose 
$1$-cell component is a \emph{left adjoint}. 
In this sense, cofunctors are dual to functors, and may be 
unified under the notion of two-dimensional partial map, as in 
\cite[Appendix]{LS02}, or Mealy morphism, as in \cite{Par12}, 
however this relationship is not pursued further in this paper. 

An internal functor suitably equipped with an internal cofunctor is   
called an internal lens (Definition~\ref{defn:internallens}) 
between internal categories. 
When $\E = \Set$, this morphism between small categories is known as
a \emph{delta lens} (Example~\ref{ex:deltalens}) introduced in the paper 
\cite{DXC11}.
The general notion of a lens, whose name derives from \emph{focusing} 
on a database view, now covers many different structures in the 
literature (see \cite{JR16} for some examples) since the influential 
paper \cite{FGMPS07}. 
Delta lenses were originally defined in computer science as an algebraic 
framework for \emph{bidirectional model transformations}, 
and their relationship to split opfibrations is important for 
applications such as \emph{least-change view updating} \cite{JRW12}.
This paper is motivated by the goal to better understand the 
precise relationship between delta lenses and split opfibrations, 
through generalising the result in \cite{JR13} that split opfibrations 
are delta lenses whose lifts are opcartesian.

Internal split opfibrations are defined as internal lenses with 
a certain property, expressed as a pullback condition 
(Definition~\ref{defn:internalsplitopf}).
In the recent paper \cite[Section~6]{AU17}, split opfibrations were 
instead defined via additional structure on a delta lens, using the 
syntax of directed containers from \cite{ACU14}.
Generalising the approach in \cite{AU17}, internal split opfibrations 
may also be described via additional structure on an internal lens,  
and this characterisation is shown 
(Proposition~\ref{prop:internalsplitopf}) to be equivalent to 
Definition~\ref{defn:internalsplitopf}, while also specialising to the
familiar definition in the $\E = \Set$ case. 

Given the category $\LENS(\E)$ of internal categories and internal 
lenses, there is an identity-on-objects functor 
$U \colon \LENS(\E) \rightarrow \CAT(\E)$ which forgets the internal
cofunctor structure. 
The right d\'{e}calage construction (Definition~\ref{defn:rightdec}) 
is a comonad on $\CAT(\E)$, and it is natural to ask if this comonad 
restricts along $U$ to an endofunctor on $\LENS(\E)$. 
The main result of this paper (Theorem~\ref{thm:main}) characterises 
internal split opfibrations as exactly those internal lenses for which
the right d\'{e}calage construction yields an internal lens. 

A classical result regarding split opfibrations is that every morphism 
in the domain factorises uniquely into a chosen opcartesian lift 
followed by a ``vertical'' morphism in the fibre 
(Example~\ref{ex:exchi}), 
thus yielding a strict factorisation system 
(Definition~\ref{defn:strictfact}) as in \cite{RW02}. 
This strict factorisation system may be generalised to internal lenses 
(Proposition~\ref{prop:isostructure}) in the form of additional 
structure. 
Surprisingly, the property of an internal lens having this strict 
factorisation system provides another characterisation of 
internal split opfibrations (Theorem~\ref{thm:johnstone}). 
This result also reveals an equivalence between the definition of
internal split opfibration in this paper and the explicit
axiomatisation via structure on an internal functor 
appearing as an exercise in \cite[Chapter~2]{Joh77}.

%----------------------------------------------------------------------%
% Subsection: Structure of the paper                                   %
%----------------------------------------------------------------------%
\subsection*{Structure of the paper}
Section~\ref{sec:background} recalls the relevant concepts from 
internal category theory for the convenience of the reader, and 
does not contain any original results. 
Section~\ref{sec:internallenses} introduces internal lenses,
and includes basic properties and examples. 
Section~\ref{sec:decalage} provides a new definition of internal 
split opfibrations and their characterisation via d\'{e}calage. 
Section~\ref{sec:factorisation} explores the connection to strict 
factorisation systems, and the earlier axiomatisation of Johnstone. 
Section~\ref{sec:conclusion} provides some concluding remarks.

%----------------------------------------------------------------------%
% Subsection: Acknowledgements                                         %
%----------------------------------------------------------------------%
\subsection*{Acknowledgements} 
The author would like to thank Michael Johnson and Stephen Lack for 
their suggestions and helpful advice on this work. 
The author is also grateful for the feedback from the anonymous 
referee. 
The key ideas on internal lenses appeared in an extended 
abstract \cite{Cla19} which was presented during the ACT2019 conference.

%----------------------------------------------------------------------%
% Section 2: Internal categories, functors, and cofunctors             %
%----------------------------------------------------------------------%
\section{Internal categories, functors, and cofunctors}
\label{sec:background}

This section reviews the background material on internal categories and 
functors, in a fixed category $\E$ with pullbacks, and establishes the 
notation for the rest of the paper.  
This section also adapts the definition of an internal cofunctor from 
\cite{Agu97} to this setting, and includes an internal version of the 
result \cite[Theorem~5.8]{HM93} which states that every cofunctor may be 
represented as a span of functors, with left leg a discrete opfibration 
and right leg an identity-on-objects functor.

\begin{defn}
\label{defn:internalcat}
An \emph{internal category} $A$ is a diagram,
\begin{equation*}
\begin{tikzcd}[column sep = large]
A_{0}
\arrow[r, "i_{0}" description]
\arrow[from = r, "d_{0}" description, shift right = 4]
\arrow[from = r, "d_{1}" description, shift left = 4]
& A_{1}
\arrow[r, "i_{0}" description, shift left = 4]
\arrow[r, "i_{1}" description, shift right = 4]
\arrow[from = r, "d_{0}" description, shift right = 8]
\arrow[from = r, "d_{1}" description]
\arrow[from = r, "d_{2}" description, shift left = 8]
& A_{2}
\arrow[from = r, "d_{0}" description, shift right = 6]
\arrow[from = r, "d_{1}" description, shift right = 2]
\arrow[from = r, "d_{2}" description, shift left = 2]
\arrow[from = r, "d_{3}" description, shift left = 6]
& A_{3}
\end{tikzcd}
\end{equation*}
where the objects $A_{2}$ and $A_{3}$ are defined by the pullbacks,
\begin{equation}
\label{eqn:cat-pullbacks}
\begin{tikzcd}[row sep = small, column sep = small]
& A_{2} 
\arrow[ld, "d_{0}"']
\arrow[rd, "d_{2}"]
\arrow[dd, phantom, "\lrcorner" rotate = -45, very near start]
& \\
A_{1}
\arrow[rd, "d_{1}"']
& & A_{1}
\arrow[ld, "d_{0}"]
\\
& A_{0} &
\end{tikzcd}
\qquad \qquad 
\begin{tikzcd}[row sep = small, column sep = small]
& A_{3} 
\arrow[ld, "d_{0}"']
\arrow[rd, "d_{3}"]
\arrow[dd, phantom, "\lrcorner" rotate = -45, very near start]
& \\
A_{2}
\arrow[rd, "d_{2}"']
& & A_{2}
\arrow[ld, "d_{0}"]
\\
& A_{1} &
\end{tikzcd}
\end{equation}
where the \emph{identity map} $i_{0} \colon A_{0} \rightarrow A_{1}$ 
and the \emph{composition map} $d_{1} \colon A_{2} \rightarrow A_{1}$ 
satisfy the commutative diagrams,
\begin{equation}
\label{eqn:identity-comp}
\begin{tikzcd}
& A_{0}
\arrow[ld, "1"']
\arrow[rd, "1"]
\arrow[d, "i_{0}"]
& \\
A_{0}
& A_{1}
\arrow[l, "d_{0}"]
\arrow[r, "d_{1}"']
& A_{0}
\end{tikzcd}
\qquad \qquad
\begin{tikzcd}
A_{1}
\arrow[d, "d_{0}"']
& A_{2}
\arrow[l, "d_{0}"']
\arrow[r, "d_{2}"]
\arrow[d, "d_{1}"]
& A_{1}
\arrow[d, "d_{1}"]
\\
A_{0}
& A_{1}
\arrow[l, "d_{0}"]
\arrow[r, "d_{1}"']
& A_{0}
\end{tikzcd}
\end{equation}
and satisfy the \emph{unitality} and \emph{associativity} 
axioms given by the commutative diagrams: 
\begin{equation}
\label{eqn:unit-assoc}
\begin{tikzcd}[row sep = small, column sep = small]
& A_{1}
\arrow[ld, "i_{0}"']
\arrow[rd, "i_{1}"]
\arrow[dd, "1"]
& \\
A_{2}
\arrow[rd, "d_{1}"']
& & A_{2}
\arrow[ld, "d_{1}"]
\\
& A_{1} & 
\end{tikzcd}
\qquad \qquad
\begin{tikzcd}[row sep = small, column sep = small]
& A_{3}
\arrow[ld, "d_{1}"']
\arrow[rd, "d_{2}"]
& \\
A_{2}
\arrow[rd, "d_{1}"']
& & A_{2}
\arrow[ld, "d_{1}"]
\\
& A_{0} & 
\end{tikzcd}
\end{equation}
An internal category is often depicted by its underlying directed graph 
consisting of the \emph{object of objects} $A_{0}$, 
the \emph{object of morphisms} $A_{1}$, 
the \emph{domain map} $d_{0} \colon A_{1} \rightarrow A_{0}$, and the
\emph{codomain map} $d_{1} \colon A_{1} \rightarrow A_{0}$.
The morphisms $i_{0}, i_{1} \colon A_{1} \rightarrow A_{2}$ 
and $d_{1}, d_{2} \colon A_{3} \rightarrow A_{2}$ 
appearing in \eqref{eqn:unit-assoc} 
are defined using the universal property of the pullback. 
\end{defn}

\begin{defn}
\label{defn:internalfun}
Let $A$ and $B$ be internal categories. 
An \emph{internal functor} $f \colon A \rightarrow B$ is a pair of 
morphisms, 
\begin{equation*}
	f_{0} \colon A_{0} \longrightarrow B_{0}
	\qquad \qquad
	f_{1} \colon A_{1} \longrightarrow B_{1}
\end{equation*}
satisfying the commutative diagrams for a directed graph homomorphism, 
\begin{equation}
\label{eqn:graphhomo}
\begin{tikzcd}
A_{0}
\arrow[d, "f_{0}"']
& A_{1}
\arrow[l, "d_{0}"']
\arrow[r, "d_{1}"]
\arrow[d, "f_{1}"]
& A_{0}
\arrow[d, "f_{0}"]
\\
B_{0}
& B_{1}
\arrow[l, "d_{0}"]
\arrow[r, "d_{1}"']
& B_{0}
\end{tikzcd}
\end{equation}
and which respect the identity and composition maps: 
\begin{equation}
\label{eqn:funidcomp}
\begin{tikzcd}
A_{0}
\arrow[r, "i_{0}"]
\arrow[d, "f_{0}"']\
& A_{1}
\arrow[d, "f_{1}"]
\\
B_{0}
\arrow[r, "i_{0}"']
& B_{1}
\end{tikzcd}
\qquad \qquad
\begin{tikzcd}
A_{2}
\arrow[r, "d_{1}"]
\arrow[d, "f_{2}"']\
& A_{1}
\arrow[d, "f_{1}"]
\\
B_{2}
\arrow[r, "d_{1}"']
& B_{1}
\end{tikzcd}
\end{equation}
The morphism $f_{2} \colon A_{2} \rightarrow B_{2}$ is defined 
using the universal property of the pullback. 
\end{defn}

\begin{remark}
\label{rem:CatE}
Let $\CAT(\E)$ denote the category of internal categories and internal 
functors in a fixed category $\E$ with pullbacks. 
It is well-known that $\CAT(\E)$ has pullbacks which are computed 
component-wise. 
\end{remark}

There are two classes of internal functors which are of particular 
interest, and the following lemma will be useful for many of the main
results. 

\begin{defn}
\label{defn:iso-on-obj}
An internal functor $f \colon A \rightarrow B$ is called 
\emph{isomorphism-on-objects} if the morphism $f_{0}$ is an isomorphism.
In particular, if $f_{0} = 1$ the functor is called 
\emph{identity-on-objects}.
\end{defn}

\begin{defn}
\label{defn:internaldopf}
An \emph{internal discrete opfibration} is an internal functor 
$f \colon A \rightarrow B$ such that the following commutative diagram,
appearing in \eqref{eqn:graphhomo}, is a pullback:
\begin{equation}
\label{eqn:discopf}
\begin{tikzcd}[row sep = small, column sep = small]
& A_{1}
\arrow[ld, "d_{0}"']
\arrow[rd, "f_{1}"]
& \\
A_{0}
\arrow[rd, "f_{0}"']
& & B_{1}
\arrow[ld, "d_{0}"]
\\
& B_{0} &
\end{tikzcd}
\end{equation}
Dually, an \emph{internal discrete fibration} is an internal functor 
$f \colon A \rightarrow B$ such that the right-hand square 
in \eqref{eqn:graphhomo} is a pullback. 
\end{defn}

\begin{lemma}
\label{lem:important}
Let $\mathcal{C}$ denote the class of isomorphism-on-objects internal 
functors (resp. internal discrete opfibrations). 
Then $\mathcal{C}$ satisfies the following conditions: 
\begin{enumerate}[label={(\roman*)}, font=\normalfont, itemsep=+1ex]
\item $\mathcal{C}$ is closed under composition;
\item $\mathcal{C}$ contains the isomorphisms;
\item $\mathcal{C}$ is stable under pullback by internal functors;
\item for all composable pairs of internal functors 
$f \colon A \rightarrow B$ and $g \colon B \rightarrow C$, 
if both $g$ and $g \circ f$ are in $\mathcal{C}$, 
then $f$ is in $\mathcal{C}$. 
\end{enumerate}
\end{lemma}

\begin{remark}
\label{rem:discopf}
An internal discrete opfibration is an internal functor with a certain 
\emph{property}, however this property may also be specified by 
\emph{unique structure}. 
Thus an internal discrete opfibration is equivalent to an internal 
functor equipped with a morphism,
\[
	\varphi_{1} \colon A_{0} \times_{B_{0}} B_{1} \longrightarrow A_{1}
\]
making the following diagrams commute:  
\begin{equation*}
\begin{tikzcd}[column sep = small]
& A_{0} \times_{B_{0}} B_{1}
\arrow[ld, "\pi_{0}"']
\arrow[d, "\varphi_{1}"]
\arrow[rd, "\pi_{1}"]
& \\
A_{0}
& A_{1}
\arrow[l, "d_{0}"]
\arrow[r, "f_{1}"']
& B_{1}
\end{tikzcd}
\qquad \qquad
\begin{tikzcd}[column sep = tiny]
& A_{0} \times_{B_{0}} B_{1}
\arrow[from=ld, "{\langle d_{0},\, f_{1} \rangle}"]
\arrow[rd, "\varphi_{1}"]
& \\
A_{1}
\arrow[rr, "1"']
& & A_{1}
\end{tikzcd}
\end{equation*}
This relationship between property and structure will later be important 
when characterising internal split opfibrations.
\end{remark}

\begin{example}[$\E = \Set$]
\label{ex:discopf}
A discrete opfibration is a functor $f \colon A \rightarrow B$ between 
small categories such that for each pair 
$(a \in A, u \colon fa \rightarrow b \in B)$ there exists a 
\emph{unique} morphism $\varphi(a, u) \colon a \rightarrow a'$ in $A$ 
such that $f\varphi(a, u) = u$. 
\begin{equation*}
\begin{tikzcd}
A
\arrow[d, "f"']
& a
\arrow[d, phantom, "\vdots"]
\arrow[r, "{\varphi(a, u)}"]
& a'
\arrow[d, phantom, "\vdots"]
\\
B
& f a
\arrow[r, "u"]
& b
\end{tikzcd}
\end{equation*}
Note that the codomain $a'$ of the morphism $\varphi(a, u)$ is also 
a function of the pair $(a, u)$, and this may be made explicit through 
writing $a' = p(a, u)$.
\end{example}

Internal cofunctors will now be introduced as a kind of morphism between 
internal categories. 
They generalise both internal discrete opfibrations and
isomorphism-on-objects internal functors. 
Since the definition of an internal cofunctor is more involved than that
of an internal functor, it is beneficial to first introduce the 
definition for small categories. 

\begin{example}[$\E = \Set$]
\label{ex:cofunctor}
A \emph{cofunctor} $\varphi \colon B \nrightarrow A$ between small 
categories consists of a function on objects 
$a \in A \, \mapsto \, \varphi a \in B$  
together with a function on pairs, 
\[
	(a \in A, u \colon \varphi a \rightarrow b \in B) 
	\qquad \longmapsto \qquad
	\varphi(a, u) \colon a \rightarrow p(a, u)
\]
satisfying the axioms: 
\begin{enumerate}[(1), itemsep=+1ex]
\item $\varphi p(a, u) = b$
\item $\varphi(a, 1_{\varphi a}) = 1_{a}$
\item $\varphi(a, v \circ u) = \varphi(p(a, u), v) \circ \varphi(a, u)$
\end{enumerate}
A cofunctor may be understood as a functorial 
\emph{lifting} of arrows from $B$ to $A$: 
\begin{equation*}
\begin{tikzcd}
A
& a
\arrow[d, phantom, "\vdots"]
\arrow[r, "{\varphi(a, u)}"]
& p(a, u) 
\arrow[d, phantom, "\vdots"]
\\
B
\arrow[u, "\varphi"{xshift=-3pt}, "/"{anchor=center,sloped}]
& \varphi a
\arrow[r, "u"]
& b
\end{tikzcd}
\end{equation*}
Note that the codomain $p(a, u)$ of the morphism $\varphi(a, u)$ 
is given explicitly as a function of the pair $(a, u)$. 
\end{example}

\begin{defn}
\label{defn:internalcofun}
Let $A$ and $B$ be internal categories. 
An \emph{internal cofunctor} $\varphi \colon B \nrightarrow A$ 
is a pair of morphisms,
\begin{equation*}
	\varphi_{0} \colon A_{0} \longrightarrow B_{0}
	\qquad \qquad
	\varphi_{1} \colon \Lambda_{1} \longrightarrow A_{1}
\end{equation*}
with the objects $\Lambda_{1}$ and $\Lambda_{2}$ defined by the 
pullbacks,
\begin{equation}
\label{eqn:cofun-pullbacks}
\begin{tikzcd}[row sep = small, column sep = small]
& \Lambda_{1} 
\arrow[ld, "d_{0}"']
\arrow[rd, "\phibar_{1}"]
\arrow[dd, phantom, "\lrcorner" rotate = -45, very near start]
& \\
A_{0}
\arrow[rd, "\varphi_{0}"']
& & B_{1}
\arrow[ld, "d_{0}"]
\\
& B_{0} &
\end{tikzcd}
\qquad \qquad 
\begin{tikzcd}[row sep = small, column sep = small]
& \Lambda_{2} 
\arrow[ld, "d_{0}"']
\arrow[rd, "\phibar_{2}"]
\arrow[dd, phantom, "\lrcorner" rotate = -45, very near start]
& \\
\Lambda_{1}
\arrow[rd, "\phibar_{1}"']
& & B_{2}
\arrow[ld, "d_{0}"]
\\
& B_{1} &
\end{tikzcd}
\end{equation}
satisfying commutative diagrams with respect to the domain and 
codomain maps, 
\begin{equation}
\label{eqn:cofundomcod}
\begin{tikzcd}
& \Lambda_{1} 
\arrow[ld, "d_{0}"']
\arrow[d, "\varphi_{1}"]
\arrow[rr, "\phibar_{1}"]
& & B_{1}
\arrow[d, "d_{1}"]
\\
A_{0}
& A_{1}
\arrow[l, "d_{0}"]
\arrow[r, "d_{1}"']
& A_{0}
\arrow[r, "\varphi_{0}"']
& B_{0}
\end{tikzcd}
\end{equation}
and with respect to the identity and composition maps:
\begin{equation}
\label{eqn:cofunidcomp}
\begin{tikzcd}
A_{0}
\arrow[d, equal]
\arrow[r, "i_{0}"]
& 
\Lambda_{1}
\arrow[d, "\varphi_{1}"] \\
A_{0}
\arrow[r, "i_{0}"']
& A_{1}
\end{tikzcd}
\qquad \qquad
\begin{tikzcd}
\Lambda_{2}
\arrow[r, "d_{1}"]
\arrow[d, "\varphi_{2}"']
& \Lambda_{1}
\arrow[d, "\varphi_{1}"]
\\
A_{2}
\arrow[r, "d_{1}"']
& A_{1}
\end{tikzcd}
\end{equation}
It is useful to define the morphism 
$p_{1} = d_{1}\varphi_{1}\colon \Lambda_{1} \rightarrow A_{0}$. 
The morphisms $i_{0} \colon A_{0} \rightarrow \Lambda_{1}$, 
$d_{1} \colon \Lambda_{2} \rightarrow \Lambda_{1}$, and 
$\varphi_{2} \colon \Lambda_{2} \rightarrow A_{2}$ are defined 
using the universal property of the pullback: 
\begin{equation*}
\begin{tikzcd}[row sep = small, column sep = small]
& A_{0}
\arrow[rd, "\varphi_{0}"]
\arrow[lddd, bend right, "1"']
\arrow[dd, "i_{0}", dashed]
& \\
& & 
B_{0}
\arrow[dd, "i_{0}"]
\\[-1em]
& \Lambda_{1}
\arrow[ld, "d_{0}"']
\arrow[rd, "\phibar_{1}"]
\arrow[dd, phantom, "\lrcorner" rotate = -45, very near start]
& \\
A_{0}
\arrow[rd, "\varphi_{0}"']
& & B_{1}
\arrow[ld, "d_{0}"]
\\
& B_{0} &
\end{tikzcd}
\quad
\begin{tikzcd}[row sep = small, column sep = small]
& \Lambda_{2}
\arrow[ld, "d_{0}"']
\arrow[rd, "\phibar_{2}"]
\arrow[dd, "d_{1}", dashed]
& \\
\Lambda_{1}
\arrow[dd, "d_{0}"']
& & B_{2}
\arrow[dd, "d_{1}"]
\\[-1em]
& \Lambda_{1} 
\arrow[ld, "d_{0}"']
\arrow[rd, "\phibar_{1}"]
\arrow[dd, phantom, "\lrcorner" rotate = -45, very near start]
& \\
A_{0}
\arrow[rd, "\varphi_{0}"']
& & B_{1}
\arrow[ld, "d_{0}"]
\\
& B_{0} &
\end{tikzcd}
\quad
\begin{tikzcd}[row sep = small, column sep = small]
& \Lambda_{2}
\arrow[ld, "d_{0}"']
\arrow[rd, "\phibar_{2}"]
\arrow[dd, "p_{2}", dashed]
& \\
\Lambda_{1}
\arrow[dd, "p_{1}"']
& & B_{2}
\arrow[dd, "d_{2}"]
\\[-1em]
& \Lambda_{1} 
\arrow[ld, "d_{0}"']
\arrow[rd, "\phibar_{1}"]
\arrow[dd, phantom, "\lrcorner" rotate = -45, very near start]
& \\
A_{0}
\arrow[rd, "\varphi_{0}"']
& & B_{1}
\arrow[ld, "d_{0}"]
\\
& B_{0} &
\end{tikzcd}
\quad
\begin{tikzcd}[row sep = small, column sep = small]
& \Lambda_{2}
\arrow[ld, "d_{0}"']
\arrow[rd, "p_{2}"]
\arrow[dd, "\varphi_{2}", dashed]
& \\
\Lambda_{1}
\arrow[dd, "\varphi_{1}"']
& & \Lambda_{1}
\arrow[dd, "\varphi_{1}"]
\\[-1em]
& A_{2} 
\arrow[ld, "d_{0}"']
\arrow[rd, "d_{2}"]
\arrow[dd, phantom, "\lrcorner" rotate = -45, very near start]
& \\
A_{1}
\arrow[rd, "d_{1}"']
& & A_{1}
\arrow[ld, "d_{0}"]
\\
& A_{0} &
\end{tikzcd}
\end{equation*}
\end{defn}

\begin{notation}
In contrast to Remark~\ref{rem:discopf}, the above definition uses the 
notation $\Lambda_{1}$ for the pullback $A_{0} \times_{B_{0}} B_{1}$ as
it is more compact (likewise for $\Lambda_{2}$). 
The above definition also uses the suggestive notation 
$d_{0} \colon \Lambda_{1} \rightarrow A_{0}$ and 
$\phibar_{1} \colon \Lambda_{1} \rightarrow B_{1}$ 
for the pullback projections 
$\pi_{0} \colon A_{0} \times_{B_{0}} B_{1} \rightarrow A_{0}$ and 
$\pi_{1} \colon A_{0} \times_{B_{0}} B_{1} \rightarrow B_{1}$, 
respectively, as this will be particularly useful for the following 
results. 
\end{notation}

\begin{remark}
\label{rem:cofunctorcat}
Let $\COF(\E)$ denote the category of internal categories and internal 
cofunctors in a fixed category $\E$ with pullbacks. 
Given internal cofunctors $\gamma \colon C \nrightarrow B$ and 
$\varphi \colon B \nrightarrow A$, the composite internal cofunctor 
$\varphi \circ \gamma \colon C \nrightarrow A$ consists of the pair of
morphisms: 
\begin{equation}
\label{eqn:compositecofun}
\gamma_{0} \varphi_{0} \colon A_{0} \longrightarrow C_{0}
\qquad \qquad
\varphi_{1} \langle \pi_{0}, \gamma_{1}(\varphi_{0} \times 1) \rangle 
\colon A_{0} \times_{C_{0}} C_{1} \longrightarrow A_{1}
\end{equation}
\end{remark}

\begin{example}
An internal discrete opfibration $A \rightarrow B$ induces an internal
cofunctor $B \nrightarrow A$, while an isomorphism-on-objects internal
functor $A \rightarrow B$ induces an internal cofunctor 
$A \nrightarrow B$. 
\end{example}

Let $\SPAN_{\mathsf{iso}}(\CAT(\E))$ denote the \emph{category} whose 
objects are internal categories and whose morphisms are 
\emph{isomorphism classes} of spans of internal functors. 
There exists a faithful, identity-on-objects functor 
$\COF(\E) \rightarrow \SPAN_{\mathsf{iso}}(\CAT(\E))$ which assigns each
internal cofunctor to an isomorphism class of spans, which have left leg
an internal discrete opfibration and right leg an isomorphism-on-objects 
internal functor. 
In practice, an internal cofunctor will always be identified with a 
chosen representative of this isomorphism class, whose right leg is an 
identity-on-objects internal functor. 
This representative span for a cofunctor will now be constructed.

\begin{lemma}
\label{lem:cofunctor}
Given an internal cofunctor $\varphi \colon B \nrightarrow A$, there 
exists an internal category, denoted $\Lambda$, defined by the diagram: 
\begin{equation*}
\begin{tikzcd}[column sep = large]
A_{0}
\arrow[r, "i_{0}" description]
\arrow[from = r, "d_{0}" description, shift right = 4]
\arrow[from = r, "p_{1}" description, shift left = 4]
& \Lambda_{1}
\arrow[from = r, "d_{0}" description, shift right = 4]
\arrow[from = r, "d_{1}" description]
\arrow[from = r, "p_{2}" description, shift left = 4]
& \Lambda_{2}
\end{tikzcd}
\end{equation*}
\end{lemma}
\begin{proof}
We prove that axiom \eqref{eqn:identity-comp} for an internal category
is satisfied, and suppress details required to prove 
\eqref{eqn:unit-assoc}.

First, to show $\Lambda_{2}$ is the pullback of the morphisms 
$d_{0}, p_{1} \colon \Lambda_{1} \rightarrow A_{0}$, 
consider the following diagrams which are equal by the construction of 
$p_{2} \colon \Lambda_{2} \rightarrow \Lambda_{1}$ in 
Definition~\ref{defn:internalcofun}:
\begin{equation*}
\begin{tikzcd}
\Lambda_{2}
\arrow[r, "p_{2}"]
\arrow[d, "d_{0}"']
& \Lambda_{1}
\arrow[d, "d_{0}"']
\arrow[r, "\phibar_{1}"]
\arrow[rd, phantom, "\lrcorner" very near start]
& B_{1}
\arrow[d, "d_{0}"]
\\
\Lambda_{1}
\arrow[r, "p_{1}"']
& A_{0}
\arrow[r, "\varphi_{0}"']
& B_{0}
\end{tikzcd}
\qquad = \qquad
\begin{tikzcd}
\Lambda_{2}
\arrow[r, "\phibar_{2}"]
\arrow[d, "d_{0}"']
\arrow[rd, phantom, "\lrcorner" very near start]
& B_{2}
\arrow[d, "d_{0}"']
\arrow[r, "d_{2}"]
\arrow[rd, phantom, "\lrcorner" very near start]
& B_{1}
\arrow[d, "d_{0}"]
\\
\Lambda_{1}
\arrow[r, "\phibar_{1}"']
& B_{1}
\arrow[r, "d_{1}"']
& B_{0}
\end{tikzcd}
\end{equation*}
Using the pullback pasting lemma, the remaining square must be a 
pullback, as required.

To show the identity map $i_{0} \colon A_{0} \rightarrow \Lambda_{1}$
and the composition map 
$d_{1} \colon \Lambda_{2} \rightarrow \Lambda_{1}$ are well-defined, 
first notice by their construction in 
Definition~\ref{defn:internalcofun} that the following diagrams commute:
\begin{equation*}
\begin{tikzcd}
& A_{0}
\arrow[d, "i_{0}"]
\arrow[ld, "1"']
\\
A_{0}
& \Lambda_{1}
\arrow[l, "d_{0}"]
\end{tikzcd}
\qquad \qquad
\begin{tikzcd}
\Lambda_{1}
\arrow[d, "d_{0}"']
& \Lambda_{2}
\arrow[l, "d_{0}"']
\arrow[d, "d_{1}"]
\\
A_{0}
& \Lambda_{1}
\arrow[l, "d_{0}"]
\end{tikzcd}
\end{equation*}
The counterparts to the above diagrams are obtained by pasting as 
follows: 
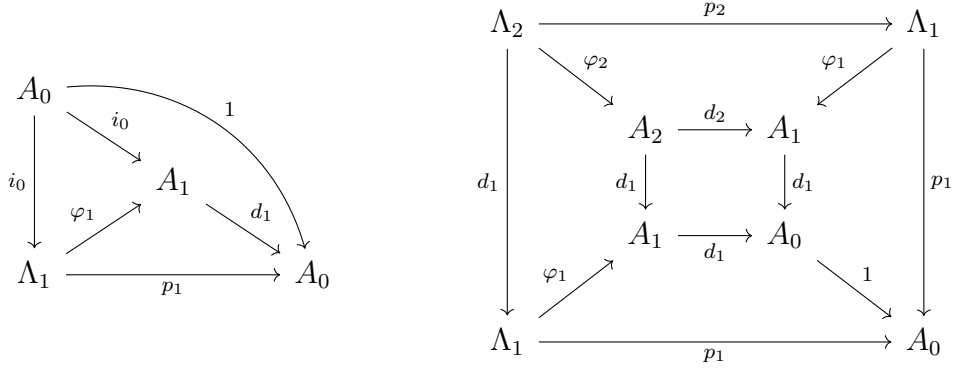
\begin{equation*}
\begin{tikzcd}[row sep = scriptsize]
A_{0}
\arrow[dd, "i_{0}"']
\arrow[rd, "i_{0}"]
\arrow[rrdd, bend left = 40, "1"]
& & \\
& A_{1}
\arrow[rd, "d_{1}"]
& \\
\Lambda_{1}
\arrow[ru, "\varphi_{1}"]
\arrow[rr, "p_{1}"']
& & A_{0}
\end{tikzcd}
\qquad \qquad
\begin{tikzcd}
\Lambda_{2}
\arrow[rd, "\varphi_{2}"]
\arrow[ddd, "d_{1}"']
\arrow[rrr, "p_{2}"]
& & & \Lambda_{1}
\arrow[ld, "\varphi_{1}"']
\arrow[ddd, "p_{1}"]
\\
& A_{2}
\arrow[d, "d_{1}"']
\arrow[r, "d_{2}"]
& A_{1}
\arrow[d, "d_{1}"]
& \\
&
A_{1}
\arrow[r, "d_{1}"']
& A_{0}
\arrow[rd, "1"]
& \\
\Lambda_{1}
\arrow[ru, "\varphi_{1}"]
\arrow[rrr, "p_{1}"']
& & & A_{0}
\end{tikzcd}
\end{equation*}
Thus axiom \eqref{eqn:identity-comp} is satisfied as required. 
\end{proof}

\begin{example}[$\E = \Set$]
\label{ex:cofuncat}
For a cofunctor $\varphi \colon B \nrightarrow A$ between small 
categories as in Example~\ref{ex:cofunctor}, the category $\Lambda$ has 
objects $a \in A$ and morphisms given by pairs depicted: 
\begin{equation*}
\begin{tikzcd}[column sep = large]
a 
\arrow[r, "{(a,\, u)}"]
& p(a, u)
\end{tikzcd}
\end{equation*}
Thus $\Lambda$ may be understood as the category of \emph{chosen lifts}
in $A$ with respect to $\varphi$. 
\end{example}

\begin{proposition}
\label{prop:cofunctor}
Given an internal cofunctor $\varphi \colon B \nrightarrow A$ there is a 
span of internal functors, 
\begin{equation}
\begin{tikzcd}[column sep = small]
& \Lambda
\arrow[ld, "\phibar"']
\arrow[rd, "\varphi"]
& \\
B 
& & A
\end{tikzcd}
\end{equation}
with left leg an internal discrete opfibration and 
right leg an identity-on-objects internal functor. 
\end{proposition}
\begin{proof}
The internal discrete opfibration $\phibar \colon \Lambda \rightarrow B$ 
is determined by the pair of morphisms $(\varphi_{0}, \phibar_{1})$ 
which satisfy the commutative diagrams: 
\begin{equation*}
\begin{tikzcd}
A_{0}
\arrow[d, "\varphi_{0}"']
& \Lambda_{1}
\arrow[l, "d_{0}"']
\arrow[r, "p_{1}"]
\arrow[d, "\phibar_{1}"]
\arrow[ld, phantom, "\llcorner", very near start]
& A_{0}
\arrow[d, "\varphi_{0}"]
\\
B_{0}
& B_{1}
\arrow[l, "d_{0}"]
\arrow[r, "d_{1}"']
& B_{0}
\end{tikzcd}
\qquad 
\begin{tikzcd}
A_{0}
\arrow[d, "\varphi_{0}"']
\arrow[r, "i_{0}"]
& \Lambda_{1}
\arrow[d, "\phibar_{1}"]
\\
B_{0}
\arrow[r, "i_{0}"']
& B_{1}
\end{tikzcd}
\qquad
\begin{tikzcd}
\Lambda_{2}
\arrow[d, "\phibar_{2}"']
\arrow[r, "d_{1}"]
& \Lambda_{1}
\arrow[d, "\phibar_{1}"]
\\
B_{2}
\arrow[r, "d_{1}"']
& B_{1}
\end{tikzcd}
\end{equation*}
The internal identity-on-objects functor 
$\varphi \colon \Lambda \rightarrow A$ is determined by the pair of 
morphisms $(1_{A_{0}}, \varphi_{1})$ which satisfy the commutative 
diagrams: 
\begin{equation*}
\begin{tikzcd}
A_{0}
\arrow[d, "1"']
& \Lambda_{1}
\arrow[l, "d_{0}"']
\arrow[r, "p_{1}"]
\arrow[d, "\varphi_{1}"]
& A_{0}
\arrow[d, "1"]
\\
A_{0}
& A_{1}
\arrow[l, "d_{0}"]
\arrow[r, "d_{1}"']
& A_{0}
\end{tikzcd}
\qquad
\begin{tikzcd}
A_{0}
\arrow[d, "1"']
\arrow[r, "i_{0}"]
& \Lambda_{1}
\arrow[d, "\varphi_{1}"]
\\
A_{0}
\arrow[r, "i_{0}"']
& A_{1}
\end{tikzcd}
\qquad
\begin{tikzcd}
\Lambda_{2}
\arrow[d, "\varphi_{2}"']
\arrow[r, "d_{1}"]
& \Lambda_{1}
\arrow[d, "\varphi_{1}"]
\\
A_{2}
\arrow[r, "d_{1}"']
& A_{1}
\end{tikzcd}
\end{equation*}
These commutative diagrams all appear in 
Definition~\ref{defn:internalcofun}, thus the proof is complete. 
\end{proof}

The above proposition shows that every internal cofunctor may be 
understood as a chosen representative span of an isomorphism class in 
$\SPAN_{\mathsf{iso}}(\CAT(\E))$. 
Given an arbitrary span in a suitable isomorphism class, there is a 
converse which states that there exists a corresponding internal 
cofunctor. 

\begin{proposition}
\label{prop:converse}
Given a span of internal functors, 
\begin{equation*}
\begin{tikzcd}[column sep = small]
& X
\arrow[ld, "h"']
\arrow[rd, "g"]
& \\
B 
& & A
\end{tikzcd}
\end{equation*}
with left leg a discrete opfibration and 
right leg an isomorphism-on-objects functor, 
there is an internal cofunctor $B \nrightarrow A$. 
\end{proposition}
\begin{proof}
Since $g \colon X \rightarrow A$ is isomorphism-on-objects, 
define the composite morphism:
\begin{equation*}
\begin{tikzcd}[column sep = small, row sep = small]
A_{0}
\arrow[rd, "g_{0}^{-1}"']
\arrow[rr, "\varphi_{0}"]
& & B_{0}
\\
& X_{0}
\arrow[ru, "h_{0}"']
&
\end{tikzcd}
\end{equation*}
This morphism may be used to define the pullback $\Lambda_{1}$ as in 
\eqref{eqn:cofun-pullbacks}. 
Since $h \colon X \rightarrow B$ is an internal discrete opfibration, 
there exists a morphism 
$g_{0}^{-1} \times 1 \colon \Lambda_{1} \rightarrow X_{1}$ 
using the universal property of the pullback, which may be used to 
define the composite morphism:
\begin{equation*}
\begin{tikzcd}[column sep = small, row sep = small]
\Lambda_{1}
\arrow[rd, "g_{0}^{-1} \times 1"']
\arrow[rr, "\varphi_{1}"]
& & A_{1}
\\
& X_{1}
\arrow[ru, "g_{1}"']
&
\end{tikzcd}
\end{equation*}
The corresponding internal cofunctor $\varphi \colon B \nrightarrow A$ 
is constructed from the pair of morphisms $(\varphi_{0}, \varphi_{1})$. 
We suppress the details required to show that the axioms for a cofunctor 
are satisfied. 
\end{proof}

Representing internal cofunctors as spans of internal 
functors also shows how the apparently complicated composition in 
$\COF(\E)$ given by \eqref{eqn:compositecofun} may be understood simply
as span composition:
\begin{equation}
\label{eqn:cofunctorspancomp}
\begin{tikzcd}[row sep = small, column sep = small]
& &[-1em] \Omega \times_{B} \Lambda
\arrow[ld]
\arrow[rd]
\arrow[dd, phantom, "\lrcorner" rotate = -45, very near start]
&[-1em] & \\
& \Omega 
\arrow[ld, "\overline{\gamma}"']
\arrow[rd, "\gamma"]
& & \Lambda
\arrow[ld, "\phibar"']
\arrow[rd, "\varphi"]
& \\
C
& & B
& & A
\end{tikzcd}
\end{equation}
This composition is well-defined by Lemma~\ref{lem:important},
and using Proposition~\ref{prop:converse} it may be shown that the 
construction $\COF(\E) \rightarrow \SPAN_{\mathsf{iso}}(\CAT(\E))$ is 
functorial.

%----------------------------------------------------------------------%
% Section 3: Internal lenses                                           %
%----------------------------------------------------------------------%
\section{Internal lenses}
\label{sec:internallenses}

This section introduces internal lenses between internal 
categories, in a fixed category $\E$ with pullbacks, generalising 
the concept of a delta lens from \cite{DXC11}. 
Internal lenses may be represented as commutative diagrams of internal 
functors, and are closed under composition and stable under pullback. 

\begin{defn}
\label{defn:internallens}
Let $A$ and $B$ be internal categories. 
An \emph{internal lens} $(f, \varphi) \colon A \br B$ 
consists of an internal functor $f \colon A \rightarrow B$ and 
an internal cofunctor $\varphi \colon B \nrightarrow A$ satisfying 
the commutative diagrams:
\begin{equation}
\label{eqn:internallens}
\begin{tikzcd}[column sep = small, row sep = small]
& A_{0}
\arrow[ld, equal]
\arrow[rd, "\varphi_{0}"]
& \\
A_{0}
\arrow[rr, "f_{0}"']
& & B_{0}
\end{tikzcd}
\qquad \qquad
\begin{tikzcd}[column sep = small, row sep = small]
& \Lambda_{1}
\arrow[ld, "\varphi_{1}"']
\arrow[rd, "\phibar_{1}"]
& \\
A_{1}
\arrow[rr, "f_{1}"']
& & B_{1}
\end{tikzcd}
\end{equation}
\end{defn}

\begin{remark}
\label{rem:lenscat}
Let $\LENS(\E)$ denote the category of internal categories and internal 
lenses in a fixed category $\E$ with pullbacks. 
Given internal lenses $(f, \varphi) \colon A \br B$ and
$(g, \gamma) \colon B \br C$, the composite internal lens
$(g \circ f, \varphi \circ \gamma) \colon A \br C$ is given by the 
composite internal functor $g \circ f \colon A \rightarrow C$ 
and the composite internal cofunctor 
$\varphi \circ \gamma \colon C \nrightarrow A$. 
One may check that \eqref{eqn:internallens} holds. 
\end{remark}

\begin{example}[$\E = \Set$, see \cite{DXC11}]
\label{ex:deltalens}
A \emph{delta lens} $(f, \varphi) \colon A \br B$ between small 
categories consists of a functor $f \colon A \rightarrow B$ together 
with a function, 
\[
	(a \in A, u \colon f a \rightarrow b \in B) 
	\qquad \longmapsto \qquad
	\varphi(a, u) \colon a \rightarrow p(a, u)
\]
satisfying the axioms: 
\begin{enumerate}[(1), itemsep=+1ex]
\item $f \varphi(a, u)= u$
\item $\varphi(a, 1_{fa}) = 1_{a}$
\item $\varphi(a, v \circ u) = \varphi(p(a, u), v) \circ \varphi(a, u)$
\end{enumerate}
A delta lens is exactly an internal lens in $\Set$, and may be 
understood as functor equipped with a functorial \emph{lifting} 
of morphisms from the codomain to the domain. 
Unlike a split opfibration however, there is no requirement for 
these lifts to be opcartesian. 
Based on the diagram for a cofunctor in Example~\ref{ex:cofunctor}, 
the following illustrates the behaviour of a delta lens: 
\begin{equation*}
\begin{tikzcd}
A
\arrow[d, "f", shift left, harpoon]
& a
\arrow[d, phantom, "\vdots"]
\arrow[r, "{\varphi(a, u)}"]
& p(a, u) 
\arrow[d, phantom, "\vdots"]
\\
B
\arrow[u, "\varphi", shift left, harpoon]
& f a
\arrow[r, "u"]
& b
\end{tikzcd}
\end{equation*}
\end{example}

\begin{example}[See \cite{FGMPS07, JRW10}] 
Let $\E$ have finite limits. 
A \emph{very well-behaved lens} consists of a pair of morphisms in $\E$, 
\[
	f \colon A \longrightarrow B
	\qquad \qquad
	p \colon A \times B \longrightarrow A
\]
satisfying three commutative diagrams: 
\begin{equation*}
\begin{tikzcd}[column sep = small]
& A \times B 
\arrow[ld, "p"']
\arrow[rd, "\pi_{1}"] & 
\\
A 
\arrow[rr, "f"']
& & B
\end{tikzcd}
\qquad
\begin{tikzcd}[column sep = small]
& A \times B 
\arrow[from=ld, "{\langle 1_{A},\, f \rangle}"]
\arrow[rd, "p"] & 
\\
A 
\arrow[rr, "1_{A}"']
& & A
\end{tikzcd}
\qquad
\begin{tikzcd}
A \times B \times B
\arrow[r, "p \times 1_{B}"]
\arrow[d, "\pi_{0, 2}"']
& A \times B
\arrow[d, "p"]
\\
A \times B 
\arrow[r, "p"']
& A
\end{tikzcd}
\end{equation*}
A very well-behaved lens is exactly an internal lens between internal 
codiscrete categories.
It consists of the internal functor between the codiscrete 
categories on $A$ and $B$, and the internal cofunctor where 
$\Lambda_{1} = A \times B$ and  
$\varphi_{1} 
= \langle \pi_{0}, p \rangle \colon A \times B \rightarrow A \times A$. 
The special case of a delta lens between small codiscrete categories was 
shown in \cite{JR16}. 
It was shown in \cite{JRW10} that the morphism 
$f \colon A \rightarrow B$ in a very well-behaved lens must be 
equivalent to a product projection with
$A \cong C \times B$.
\end{example}

\begin{example}
Let $\E$ have finite limits, and let $A$ and $B$ be internal monoids, 
considered as internal categories whose object of objects is terminal. 
An internal lens $(f, \varphi) \colon A \br B$ is exactly a monoid
homomorphism $f \colon A \rightarrow B$ with a chosen right inverse 
$\varphi \colon B \rightarrow A$.
For arbitrary internal categories, any isomorphism-on-objects internal
functor with a chosen right inverse forms an internal lens. 
\end{example}

\begin{example}
An internal lens $(f, \varphi) \colon A \br B$ is an internal discrete
opfibration if and only if 
$\varphi_{1} \colon \Lambda_{1} \rightarrow A_{1}$ is an isomorphism 
(see Remark~\ref{rem:discopf}). 
Internal lenses may be understood as being more general than internal 
discrete opfibrations while being less general than internal cofunctors. 
\end{example}

The next two results follow almost immediately from the corresponding 
results for cofunctors (Proposition~\ref{prop:cofunctor} and 
Proposition~\ref{prop:converse}) and the definition of an internal lens. 

\begin{proposition}
\label{prop:internallens}
Given an internal lens $(f, \varphi) \colon A \br B$ there is a 
a commutative diagram of internal functors, 
\begin{equation}
\label{eqn:lensrep}
\begin{tikzcd}[column sep = small]
& \Lambda
\arrow[ld, "\varphi"']
\arrow[rd, "\phibar"]
& \\
A 
\arrow[rr, "f"']
& & B
\end{tikzcd}
\end{equation}
where $\varphi$ is a faithful, identity-on-objects functor and 
$\phibar$ is a discrete opfibration. 
\end{proposition}

\begin{proposition}
\label{prop:lensconverse}
Given a commutative diagram of internal functors, 
\begin{equation*}
\begin{tikzcd}[column sep = small]
& X
\arrow[ld, "g"']
\arrow[rd, "h"]
& \\
A 
\arrow[rr, "f"']
& & B
\end{tikzcd}
\end{equation*}
where $g$ is an isomorphism-on-objects functor and $h$ is a discrete 
opfibration, there is an internal lens $A \br B$. 
\end{proposition}

The above propositions allow one to identify an internal lens with a 
diagram \eqref{eqn:lensrep}, and prove results concerning internal 
lenses using the properties of isomorphisms-on-objects internal functors
and internal discrete opfibrations stated in Lemma~\ref{lem:important}. 
In Theorem~\ref{thm:main}, internal split opfibrations will be 
characterised in this way. 

\begin{remark}
\label{rem:lensrepcomp}
Given internal lenses $(f, \varphi) \colon A \br B$ and 
$(g, \gamma) \colon B \br C$, their composite may be computed by 
constructing pullbacks in $\CAT(\E)$: 
\begin{equation*}
\begin{tikzcd}[row sep = small, column sep = small]
& &[-1em] \Lambda \times_{B} \Omega
\arrow[ld]
\arrow[rd]
\arrow[dd, phantom, "\lrcorner" rotate = -45, very near start]
&[-1em] & \\
& \Lambda 
\arrow[ld, "\varphi"']
\arrow[rd, "\phibar"]
& & \Omega
\arrow[ld, "\gamma"']
\arrow[rd, "\overline{\gamma}"]
& \\
A
\arrow[rr, "f"']
& & B
\arrow[rr, "g"']
& & C
\end{tikzcd}
\end{equation*}
\end{remark}

\begin{proposition}[Stability under pullback]
\label{prop:lenspullback}
Given a pullback square of internal functors,
\begin{equation*}
\begin{tikzcd}
A \times_{B} C
\arrow[r, "\pi_{1}"]
\arrow[d, "\pi_{0}"']
\arrow[rd, phantom, "\lrcorner" very near start]
& C
\arrow[d, "g"]
\\
A 
\arrow[r, "f"']
& B
\end{tikzcd}
\end{equation*}
if $(f, \varphi) \colon A \br B$ is an internal lens, 
then $\pi_{1}$ has the structure of an internal lens. 
\end{proposition}
\begin{proof}
Given that $(f, \varphi)$ is an internal lens, 
using Proposition~\ref{prop:internallens} we may construct the following
diagram of internal functors:
\begin{equation*}
\begin{tikzcd}
\Lambda \times_{B} C
\arrow[d, "\pi_{0}"']
\arrow[r, dashed, "\varphi \times 1"]
\arrow[rr, bend left, "\pi_{1}"]
\arrow[rd, phantom, "\lrcorner" very near start]
&
A \times_{B} C
\arrow[r, "\pi_{1}"]
\arrow[d, "\pi_{0}"']
\arrow[rd, phantom, "\lrcorner" very near start]
& C
\arrow[d, "g"]
\\
\Lambda
\arrow[r, "\varphi"']
\arrow[rr, bend right, "\phibar"']
&
A 
\arrow[r, "f"']
& B
\end{tikzcd}
\end{equation*}
By Lemma~\ref{lem:important}, both identity-on-objects functors and 
discrete opfibrations are stable under pullback along internal functors, 
thus $\varphi \times 1 \colon 
\Lambda \times_{B} C \rightarrow A \times_{B} C$ 
is identity-on-objects and 
$\pi_{1} \colon \Lambda \times_{B} C \rightarrow C$ is a discrete 
opfibration. 
By applying Proposition~\ref{prop:lensconverse}, the internal functor 
$\pi_{1} \colon A \times_{B} C \rightarrow C$ has the structure of an
internal lens. 
\end{proof}

%----------------------------------------------------------------------%
% Section 4: Internal split opfibrations and d\'{e}calage              %
%----------------------------------------------------------------------%
\section{Internal split opfibrations and d\'{e}calage}
\label{sec:decalage}

This section defines an internal split opfibration, in a fixed category 
$\E$ with pullbacks, as an internal lens with a certain property. 
This definition is shown to be equivalent to a previous 
characterisation of split opfibrations in \cite{AU17}, and 
the main theorem will further characterise internal split opfibrations 
using the right d\'{e}calage construction. 

To prepare for the main definition, 
consider an internal lens $(f, \varphi) \colon A \br B$ and construct
the following pullback: 
\begin{equation}
\label{eqn:kpullback}
\begin{tikzcd}[row sep = small, column sep = tiny]
& \Lambda_{1} \times_{A_{1}} A_{2}
\arrow[ld, "\pi_{0}"']
\arrow[rd, "\pi_{1}"]
\arrow[dd, phantom, "\lrcorner" rotate = -45, very near start]
& \\
\Lambda_{1}
\arrow[rd, "\varphi_{1}"']
& & A_{2}
\arrow[ld, "d_{0}"]
\\
& A_{1} &
\end{tikzcd}
\end{equation}

The projection 
$\pi_{1} \colon \Lambda_{1} \times_{A_{1}} A_{2} \rightarrow A_{2}$ 
may be pasted with the commutative diagram for the composition map 
in \eqref{eqn:funidcomp} to obtain:
\begin{equation}
\label{eqn:important}
\begin{tikzcd}[row sep = small, column sep = tiny]
& \Lambda_{1} \times_{A_{1}} A_{2}
\arrow[dd, "\pi_{1}"]
\arrow[lddd, bend right, "d_{1}\pi_{1}"']
\arrow[rddd, bend left, "f_{2}\pi_{1}"]
& \\
& & \\
& A_{2}
\arrow[ld, "d_{1}"']
\arrow[rd, "f_{2}"]
& \\
A_{1}
\arrow[rd, "f_{1}"']
& & B_{2}
\arrow[ld, "d_{1}"]
\\
& B_{1} &
\end{tikzcd}
\qquad = \qquad
\begin{tikzcd}[row sep = small, column sep = tiny]
& \Lambda_{1} \times_{A_{1}} A_{2}
\arrow[ld, "d_{1}\pi_{1}"']
\arrow[rd, "f_{2}\pi_{1}"]
& \\
A_{1}
\arrow[rd, "f_{1}"']
& & B_{2}
\arrow[ld, "d_{1}"]
\\
& B_{1} &
\end{tikzcd}
\end{equation}
To reiterate, the diagram \eqref{eqn:important} commutes for 
any internal lens $(f, \varphi) \colon A \br B$.

\begin{defn}
\label{defn:internalsplitopf}
An \emph{internal split opfibration} is an internal lens 
$(f, \varphi) \colon A \br B$
such that the following commutative diagram,
constructed in \eqref{eqn:important}, is a pullback:
\begin{equation*}
\begin{tikzcd}[row sep = small, column sep = tiny]
& \Lambda_{1} \times_{A_{1}} A_{2}
\arrow[ld, "d_{1}\pi_{1}"']
\arrow[rd, "f_{2}\pi_{1}"]
& \\
A_{1}
\arrow[rd, "f_{1}"']
& & B_{2}
\arrow[ld, "d_{1}"]
\\
& B_{1} &
\end{tikzcd}
\end{equation*}
\end{defn}

The above definition mirrors Definition~\ref{defn:internaldopf}, which 
states that an internal discrete opfibration is an internal functor 
$f \colon A \rightarrow B$ such that the commutative diagram 
\eqref{eqn:discopf} is a pullback. 

While an internal lens may have the \emph{property} of being an internal
split opfibration, this property may also be specified by 
\emph{unique structure}, as was done in Remark~\ref{rem:discopf} for 
internal discrete opfibrations. 
This characterisation in terms of unique structure will be described 
Proposition~\ref{prop:internalsplitopf} and allows for a straightforward 
verification that the above definition correctly generalises the 
$\E = \Set$ case. 

There is also the question of proving that internal categories and 
internal split opfibrations form a category, denoted $\SOPF(\E)$, and 
also that internal split opfibrations are stable under pullback. 
An answer is found through a natural characterisation of internal 
split opfibrations via the right d\'{e}calage comonad in 
Theorem~\ref{thm:main},
which utilises the equivalent representation of an internal lens given
by \eqref{eqn:lensrep}. 
This characterisation also provides motivation for \emph{why} the 
commuting diagram \eqref{eqn:important} is used to define an internal
split opfibration. 

Finally, one might protest that the real benefit in defining internal 
discrete opfibrations using a pullback condition is that the object of 
morphisms of the domain is endowed with the corresponding universal 
property, which is not apparent for internal split opfibrations as given
in Definition~\ref{defn:internalsplitopf}.
This is addressed in the following section, where it is shown that an 
internal split opfibration may also be characterised as an internal 
lens with a particular strict factorisation system on the domain, 
which endows the object of morphisms with a suitable universal property.  

First, to unpack the definition of an internal split opfibration 
consider the pullback:
\begin{equation}
\label{eqn:fpullback}
\begin{tikzcd}[row sep = small, column sep = tiny]
& A_{1} \times_{B_{1}} B_{2}
\arrow[ld, "\pi_{0}"']
\arrow[rd, "\pi_{1}"]
\arrow[dd, phantom, "\lrcorner" rotate = -45, very near start]
& \\
A_{1}
\arrow[rd, "f_{1}"']
& & B_{2}
\arrow[ld, "d_{1}"]
\\
& B_{1} &
\end{tikzcd}
\end{equation}
The following result characterises internal split opfibrations in 
terms of unique structure on an internal lens.
This characterisation is essentially the same as given in \cite{AU17}, 
except formulated using internal category theory rather than 
directed containers. 

\begin{proposition}
\label{prop:internalsplitopf}
An internal lens $(f, \varphi) \colon A \br B$ is an internal split 
opfibration if and only if there exists a morphism, 
\[
	\psi \colon A_{1} \times_{B_{1}} B_{2} \longrightarrow A_{1}
\]
satisfying the following four commutative diagrams: 
\begin{itemize}[$\diamond$]
\item This diagram specifies a ``domain condition'':
\begin{equation}
\label{eqn:sopfdom}
\begin{tikzcd}
A_{1} \times_{B_{1}} B_{2}
\arrow[rr, "\psi"]
\arrow[d, "d_{0} \times d_{0}"']
& & A_{1}
\arrow[d, "d_{0}"]
\\
\Lambda_{1}
\arrow[r, "\varphi_{1}"']
& A_{1}
\arrow[r, "d_{1}"']
& A_{0}
\end{tikzcd}
\end{equation}
\item This diagram specifies a ``uniqueness condition'': 
\begin{equation}
\label{eqn:sopfunique}
\begin{tikzcd}
\Lambda_{1} \times_{A_{1}} A_{2}
\arrow[r, "\pi_{1}"]
\arrow[d, "\pi_{1}"']
&[+0.5em] A_{2}
\arrow[r, "{\langle d_{1},\, f_{2} \rangle}"]
&[+0.5em] A_{1} \times_{B_{1}} B_{2}
\arrow[d, "\psi"]
\\
A_{2}
\arrow[rr, "d_{2}"']
& & A_{1}
\end{tikzcd}
\end{equation}
\item This diagram specifies a ``lifting condition'': 
\begin{equation}
\label{eqn:sopflift}
\begin{tikzcd}
A_{1} \times_{B_{1}} B_{2}
\arrow[r, "\psi"]
\arrow[d, "\pi_{1}"']
& A_{1}
\arrow[d, "f_{1}"]
\\
B_{2}
\arrow[r, "d_{2}"']
& B_{1}
\end{tikzcd}
\end{equation}
\item This diagram specifies a ``composition condition'':
\begin{equation}
\label{eqn:sopfcomp}
\begin{tikzcd}[column sep = tiny]
& A_{1} \times_{B_{1}} B_{2}
\arrow[ld, "\widehat{\psi}"']
\arrow[rd, "\pi_{0}"]
& \\
A_{2}
\arrow[rr, "d_{1}"']
& & A_{1}
\end{tikzcd}
\end{equation}
\end{itemize}
The pullback $\Lambda_{1} \times_{A_{1}} A_{2}$ is from 
\eqref{eqn:kpullback}, the pullback $A_{1} \times_{B_{1}} B_{2}$ is from
\eqref{eqn:fpullback}, and the morphism $\widehat{\psi}$ is defined 
using \eqref{eqn:sopfdom} via the universal property of the pullback: 
\begin{equation*}
\begin{tikzcd}[row sep = small, column sep = tiny]
& A_{1} \times_{B_{1}} B_{2}
\arrow[dd, dashed, "\widehat{\psi}"]
\arrow[ld, "d_{0} \times d_{0}"']
\arrow[rddd, "\psi", bend left]
& \\
\Lambda_{1}
\arrow[dd, "\varphi_{1}"']
& & \\[-1em]
& A_{2}
\arrow[ld, "d_{0}"']
\arrow[rd, "d_{2}"]
\arrow[dd, phantom, "\lrcorner" rotate = -45, very near start]
& \\
A_{1}
\arrow[rd, "d_{1}"']
& & A_{1}
\arrow[ld, "d_{0}"]
\\
& A_{0} & 
\end{tikzcd}
\end{equation*}
\end{proposition}
\begin{proof}
First notice from Definition~\ref{defn:internalsplitopf} that 
an internal lens $(f, \varphi) \colon A \br B$ is an internal split 
opfibration if and only if there is an isomorphism 
$\Lambda_{1} \times_{A_{1}} A_{2} \cong A_{1} \times_{B_{1}} B_{1}$.
For any internal lens the diagram \eqref{eqn:important} commutes, thus 
following composite morphism exists by the universal property of the 
pullback:
\begin{equation}
\label{eqn:findinverse}
\begin{tikzcd}
\Lambda_{1} \times_{A_{1}} A_{2}
\arrow[r, "\pi_{1}"]
&[+0.5em] A_{2}
\arrow[r, "{\langle d_{1},\, f_{2} \rangle}"]
&[+0.5em] A_{1} \times_{B_{1}} B_{2}
\end{tikzcd}
\end{equation}
Therefore an internal lens is an internal split opfibration if and only
if \eqref{eqn:findinverse} has an inverse. 
Given a morphism 
$\psi \colon A_{1} \times_{B_{1}} B_{2} \rightarrow A_{1}$ satisfying
axiom \eqref{eqn:sopfdom} there exists a morphism, 
\[
	\langle d_{0} \times d_{0}, \widehat{\psi} \rangle \colon 
	A_{1} \times_{B_{1}} B_{2} \longrightarrow 
	\Lambda_{1} \times_{A_{1}} A_{2}
\]
which is the required inverse if axioms \eqref{eqn:sopfunique}, 
\eqref{eqn:sopflift}, and \eqref{eqn:sopfcomp} are satisfied. 
Conversely, an inverse to \eqref{eqn:findinverse} exists exactly 
when there is a morphism $\psi$ which satisfies
the axioms in the statement of Proposition~\ref{prop:internalsplitopf}. 
The necessary diagram-chasing has been excluded. 
\end{proof}

The following example shows that the above characterisation of internal
split opfibrations yields the usual definition for $\E = \Set$, and 
explains the meaning of the names given to the axioms above. 

\begin{example}[$\E = \Set$]
\label{ex:sopf-dlens}
A delta lens $(f, \varphi) \colon A \br B$ 
(see Example~\ref{ex:deltalens}) is a split opfibration if for all 
commutative diagrams of the form, 
\begin{equation*}
\begin{tikzcd}[column sep = small, row sep = small]
f a
\arrow[rr, "u"]
\arrow[rd, "fw"']
& & b
\arrow[ld, "v"]
\\
& f a' &
\end{tikzcd}
\qquad
\text{in $B$}
\end{equation*}
for all $w \colon a \rightarrow a'$ in $A$, there exists a unique 
morphism $\psi(w, u, v) \colon p(a, u) \rightarrow a'$ such that 
$f \psi(w, u, v) = v$ and the following diagram commutes:
\begin{equation*}
\begin{tikzcd}[column sep = small, row sep = small]
a
\arrow[rr, "{\varphi(a, u)}"]
\arrow[rd, "w"']
& & p(a, u)
\arrow[ld, dashed, "{\psi(w, u, v)}"]
\\
& a' &
\end{tikzcd}
\qquad
\text{in $A$}
\end{equation*}
The domain of $\psi(w, u, v)$ is determined by \eqref{eqn:sopfdom}
and the uniqueness by \eqref{eqn:sopfunique}. 
The other two conditions on $\psi(w, u, v)$ are determined exactly by 
\eqref{eqn:sopflift} and \eqref{eqn:sopfcomp}, respectively. 
Thus $\psi$ provides a global way of stating that the chosen lifts 
given by $\varphi$ are \emph{$f$-opcartesian}.
\end{example}

\begin{example}
Every internal discrete opfibration is an internal split opfibration. 
\end{example}

There is an intrinsic way of characterising internal split opfibrations 
with respect to the right d\'{e}calage construction on $\CAT(\E)$, 
using the properties of isomorphism-on-objects internal functors and 
internal discrete opfibrations from Lemma~\ref{lem:important}.  

\begin{defn}
\label{defn:rightdec}
The \emph{right d\'{e}calage} construction is a comonad 
$D$ on $\CAT(\E)$ which assigns each internal 
functor $f \colon A \rightarrow B$ to an internal functor 
$Df \colon DA \rightarrow DB$ given by: 
\begin{equation}
\label{eqn:rightdecfunctor}
\begin{tikzcd}
A_{1}
\arrow[d, "f_{1}"']
& A_{2}
\arrow[l, "d_{1}"']
\arrow[d, "f_{2}"']
\arrow[r, "d_{2}"]
& A_{1}
\arrow[d, "f_{1}"]
\\
B_{1}
& B_{2}
\arrow[l, "d_{1}"]
\arrow[r, "d_{2}"']
& B_{1}
\end{tikzcd}
\end{equation}
The counit $\varepsilon \colon D \Rightarrow 1$ for the comonad assigns 
each internal category $A$ to an internal discrete fibration 
$\varepsilon_{A} \colon DA \rightarrow A$ given by:
\begin{equation}
\label{eqn:rightdec}
\begin{tikzcd}
A_{1}
\arrow[d, "d_{0}"']
& A_{2}
\arrow[l, "d_{1}"']
\arrow[rd, phantom, "\lrcorner" very near start]
\arrow[d, "d_{0}"']
\arrow[r, "d_{2}"]
& A_{1}
\arrow[d, "d_{0}"]
\\
A_{0}
& A_{1}
\arrow[l, "d_{0}"]
\arrow[r, "d_{1}"']
& A_{0}
\end{tikzcd}
\end{equation}
\end{defn}

\begin{example}[$\E = \Set$]
Given a small category $A$, the right d\'{e}calage $DA$ has objects 
given by arrows $w \colon a' \rightarrow a$ in $A$ 
and morphisms $u \colon w \rightsquigarrow v$ given by 
commutative triangles: 
\begin{equation*}
\begin{tikzcd}[row sep = small, column sep = small]
a'
\arrow[rd, "w"']
\arrow[rr, "u"]
& & a''
\arrow[ld, "v"]
\\
& a & 
\end{tikzcd}
\end{equation*}
That is, the right d\'{e}calage $DA$ is given by the coproduct of the 
slice categories of $A$.
The counit $\varepsilon_{A} \colon DA \rightarrow A$
sends a morphism $u \colon w \rightsquigarrow v$ in $DA$, depicted
above, to the morphism $u \colon a' \rightarrow a''$ in $A$.  
\end{example}

Given an internal lens $(f, \varphi) \colon A \br B$ with representation
in $\CAT(\E)$ given by \eqref{eqn:lensrep}, 
construct the following diagram using the counit of the right 
d\'{e}calage comonad: 
\begin{equation}
\label{eqn:rightdecmain}
\begin{tikzcd}
\Lambda \times_{A} DA
\arrow[r, "\pi_{1}"]
\arrow[d, "\pi_{0}"']
\arrow[rd, phantom, "\lrcorner" very near start]
& DA
\arrow[d, "\varepsilon_{A}"]
\arrow[r, "Df"]
& DB
\arrow[d, "\varepsilon_{B}"]
\\
\Lambda
\arrow[r, "\varphi"]
\arrow[rr, bend right, "\phibar"']
& A 
\arrow[r, "f"]
& B
\end{tikzcd}
\end{equation}
Notice that the projection $\pi_{1}$ is an identity-on-objects internal 
functor, therefore it is natural to ask: for which internal lenses 
$(f, \varphi) \colon A \br B$ is the composite $Df \circ \pi_{1}$ an 
internal discrete opfibration?  

\begin{theorem}
\label{thm:main}
An internal lens $(f, \varphi) \colon A \br B$ is an internal split
opfibration if and only if the internal functor 
$Df \circ \pi_{1} \colon \Lambda \times_{A} DA \rightarrow DB$,
defined in \eqref{eqn:rightdecmain}, is an internal discrete 
opfibration. 
\end{theorem}
\begin{proof}
The pullback $\Lambda \times_{A} DA$ is has object of objects $A_{1}$
and object of morphisms $\Lambda_{1} \times_{A_{1}} A_{2}$ as defined in
\eqref{eqn:kpullback}.
The projection $\pi_{1} \colon \Lambda \times_{A} DA \rightarrow DA$ 
is an identity-on-objects internal functor given by: 
\begin{equation*}
\begin{tikzcd}
A_{1}
\arrow[d, "1"']
& \Lambda_{1} \times_{A_{1}} A_{2}
\arrow[l, "d_{1}\pi_{1}"']
\arrow[r, "d_{2}\pi_{1}"]
\arrow[d, "\pi_{1}"]
& A_{1}
\arrow[d, "1"]
\\
A_{1}
& A_{2}
\arrow[l, "d_{1}"]
\arrow[r, "d_{2}"']
& A_{1}
\end{tikzcd}
\end{equation*}
Post-composing this projection by the internal functor 
$Df \colon DA \rightarrow DB$, defined in \eqref{eqn:rightdecfunctor}, 
yields an internal functor given by: 
\begin{equation*}
\begin{tikzcd}
A_{1}
\arrow[d, "f_{1}"']
& \Lambda_{1} \times_{A_{1}} A_{2}
\arrow[l, "d_{1}\pi_{1}"']
\arrow[r, "d_{2}\pi_{1}"]
\arrow[d, "f_{2}\pi_{1}"]
& A_{1}
\arrow[d, "f_{1}"]
\\
B_{1}
& B_{2}
\arrow[l, "d_{1}"]
\arrow[r, "d_{2}"']
& B_{1}
\end{tikzcd}
\end{equation*}
By Definition~\ref{defn:internaldopf}, 
this internal functor is an internal discrete opfibration if and only 
if the commutative diagram \eqref{eqn:important} is a pullback, 
which holds if and only if the internal lens 
$(f, \varphi) \colon A \br B$ is an internal split opfibration by 
Definition~\ref{defn:internalsplitopf}. 
\end{proof}

\begin{corollary}
\label{cor:rightdec}
If $(f, \varphi) \colon A \br B$ is an internal split opfibration, 
then the internal functor $Df \colon DA \rightarrow DB$, defined in 
\eqref{eqn:rightdecfunctor}, has the structure of an internal lens whose
internal cofunctor component consists of the morphisms: 
\[
	f_{1} \colon A_{1} \longrightarrow B_{1}
	\qquad \qquad
	\widehat{\psi} \colon A_{1} \times_{B_{1}} B_{2} 
	\longrightarrow A_{2}
\]
\end{corollary}
\begin{proof}
If $(f, \varphi) \colon A \br B$ is an internal split opfibration, 
then from Theorem~\ref{thm:main} there is a commutative diagram of 
internal functors,
\begin{equation*}
\begin{tikzcd}[column sep = tiny]
& \Lambda \times_{A} DA
\arrow[ld, "\pi_{1}"']
\arrow[rd, "Df \circ \pi_{1}"]
& \\
DA
\arrow[rr, "Df"']
& & DB
\end{tikzcd}
\end{equation*}
where $\pi_{1}$ is an identity-on-objects internal functor and 
$Df \circ \pi_{1}$ is an internal discrete opfibration.
Applying Proposition~\ref{prop:lensconverse} and 
Proposition~\ref{prop:internalsplitopf} yields the result. 
\end{proof}

Theorem~\ref{thm:main} shows that the property an internal lens must
satisfy to be an internal split opfibration is derived from the 
property of an internal functor induced from the right d\'{e}calage
comonad to be an internal discrete opfibration. 
Furthermore, Corollary~\ref{cor:rightdec} shows that the unique 
structure used to characterise internal split opfibrations in 
Proposition~\ref{prop:internalsplitopf} is equivalent to giving the 
internal functor $Df \colon DA \rightarrow DB$ a particular internal 
lens structure. 
Thus an internal split opfibration may be described by a suitable 
pair of internal lenses, making it straightforward to prove the 
well-known results that they are closed under composition and stable
under pullback (see Remark~\ref{rem:lensrepcomp} and 
Proposition~\ref{prop:lenspullback}). 

\begin{remark}
Let $\SOPF(\E)$ be the category of internal categories and internal 
split opfibrations, together with the faithful forgetful functor 
$\SOPF(\E) \rightarrow \LENS(\E)$.
By Corollary~\ref{cor:rightdec}, the right d\'{e}calage comonad 
restricts to a functor $D \colon \SOPF(\E) \rightarrow \LENS(\E)$. 
\end{remark}

%----------------------------------------------------------------------%
% Section 5:Internal split opfibrations and strict factorisation system%
%----------------------------------------------------------------------%
\section{Internal split opfibrations and strict factorisation systems} 
\label{sec:factorisation}

This section characterises an internal split opfibration via a strict 
factorisation system on the domain of an internal lens, corresponding to 
the known result in $\Set$ that every morphism in the domain factorises 
uniquely into a \emph{chosen} opcartesian lift followed by a 
\emph{vertical} morphism. 
It is shown that this equivalent formulation of an internal split 
opfibration agrees with the definition given in 
\cite[Chapter~2, Exercise~6]{Joh77}.

\begin{defn}
\label{defn:strictfact}
A \emph{strict factorisation system} $(E, M)$ on an internal category 
$A$ consists of a pair of faithful, identity-on-objects internal 
functors, 
\[
	e \colon E \longrightarrow A \qquad \qquad 
	m \colon M \longrightarrow A
\]
together with the pullback, 
\begin{equation}
\label{eqn:strictfact}
\begin{tikzcd}[column sep = tiny, row sep = small]
& E_{1} \times_{A_{0}} M_{1}
\arrow[ld, "\pi_{0}"']
\arrow[rd, "\pi_{1}"]
\arrow[dd, phantom, "\lrcorner" rotate = -45, very near start]
& \\
E_{1}
\arrow[rd, "d_{1}"']
& & M_{1}
\arrow[ld, "d_{0}"]
\\
& A_{0} & 
\end{tikzcd}
\end{equation}
such that the morphism
$d_{1}(e_{1} \times m_{1}) \colon E_{1} \times_{A_{0}} M_{1} 
\rightarrow A_{1}$
is an isomorphism.
\end{defn}

\begin{example}[$\E = \Set$]
A strict factorisation system $(E, M)$ on a small category $A$ consists
of wide subcategories $E$ and $M$ of $A$ such that every morphism 
$w \in A$ has a strictly unique factorisation $w = m \circ e$ with 
$m \in M$ and $e \in E$. 
\end{example}

Given an internal lens $(f, \varphi) \colon A \br B$, the goal is to
give a strict factorisation system $(\Lambda, V)$ on the internal 
category $A$, where the internal functor 
$\varphi \colon \Lambda \rightarrow A$ provides the left class, 
and the right class $j \colon V \rightarrow A$ is defined as follows.

\begin{defn}
\label{defn:verticalmorphisms}
Given an internal functor $f \colon A \rightarrow B$, there is a
faithful, identity-on-objects internal functor 
$j \colon V \rightarrow A$ constructed by the following pullback, 
\begin{equation}
\label{eqn:vertcat}
\begin{tikzcd}[row sep = small, column sep = small]
& V
\arrow[rd, "\widehat{f}"]
\arrow[ld, "j"']
\arrow[dd, phantom, "\lrcorner" rotate = -45, very near start]
& \\
A
\arrow[rd, "f"']
& & 
B_{0}
\arrow[ld, "i_{0}"]
\\
& B &
\end{tikzcd}
\qquad 
\rightsquigarrow
\qquad
\begin{tikzcd}[row sep = small, column sep = small]
& V_{1}
\arrow[rd, "\widehat{f}_{1}"]
\arrow[ld, "j_{1}"']
\arrow[dd, phantom, "\lrcorner" rotate = -45, very near start]
& \\
A_{1}
\arrow[rd, "f_{1}"']
& & 
B_{0}
\arrow[ld, "i_{0}"]
\\
& B_{1} &
\end{tikzcd}
\end{equation}
where $i \colon B_{0} \rightarrow B$ is the inclusion of the discrete
category of objects, and $V$ is the 
\emph{internal category of vertical morphisms} for an internal functor 
$f$. 
Explicitly, the internal functor $j \colon V \rightarrow A$ is given by: 
\begin{equation}
\begin{tikzcd}
A_{0}
\arrow[d, "1"']
& V_{1}
\arrow[l, "d_{0}"']
\arrow[r, "d_{1}"]
\arrow[d, "j_{1}"]
& A_{0}
\arrow[d, "1"]
\\
A_{0}
& A_{1}
\arrow[l, "d_{0}"]
\arrow[r, "d_{1}"']
& A_{0}
\end{tikzcd}
\end{equation}
\end{defn}

To define the strict factorisation system $(\Lambda, V)$ on $A$
for an internal lens $(f, \varphi) \colon A \br B$,
first construct the following pullback: 
\begin{equation}
\label{eqn:factorpullback}
\begin{tikzcd}[row sep = small, column sep = small]
& \Lambda_{1} \times_{A_{0}} V_{1}
\arrow[ld, "\pi_{0}"']
\arrow[dd, phantom, "\lrcorner" rotate = -45, very near start]
\arrow[rd, "\pi_{1}"]
& 
\\
\Lambda_{1}
\arrow[rd, "p_{1}"']
& & V_{1}
\arrow[ld, "d_{0}"]
\\
& A_{0} &
\end{tikzcd}
\end{equation}
The following result provides the unique structure such that
 $d_{1}(\varphi_{1} \times j_{1}) \colon 
\Lambda_{1} \times_{A_{0}} V_{1} \rightarrow A_{1}$
is an isomorphism. 

\begin{proposition}
\label{prop:isostructure}
If $(f, \varphi) \colon A \br B$ is an internal lens, 
then $(\Lambda, V)$ is a strict factorisation system on $A$ consisting 
of the internal functors $\varphi \colon \Lambda \rightarrow A$ and 
$j \colon V \rightarrow A$ 
if and only if there exists an endomorphism,
\[
	\chi \colon A_{1} \longrightarrow A_{1}
\]
satisfying the following four commutative diagrams: 
\begin{itemize}[$\diamond$]
\item This diagram specifies a ``domain condition'': 
\begin{equation}
\label{eqn:isodom}
\begin{tikzcd}
A_{1}
\arrow[rr, "\chi"]
\arrow[d, "{\langle d_{0}, f_{1} \rangle}"']
& & A_{1}
\arrow[d, "d_{0}"]
\\
\Lambda_{1}
\arrow[r, "\varphi_{1}"']
& A_{1}
\arrow[r, "d_{1}"']
& A_{0}
\end{tikzcd}
\end{equation}
\item This diagram specifies a ``uniqueness condition'': 
\begin{equation}
\label{eqn:isounique}
\begin{tikzcd}
\Lambda_{1} \times_{A_{0}} V_{1}
\arrow[d, "\pi_{1}"']
\arrow[r, "\varphi_{1} \times j_{1}"]
& A_{2}
\arrow[r, "d_{1}"]
& A_{1}
\arrow[d, "\chi"]
\\
V_{1}
\arrow[rr, "j_{1}"']
& & A_{1}
\end{tikzcd}
\end{equation}
\item This diagram specifies a ``fibre condition'':
\begin{equation}
\label{eqn:isofibre}
\begin{tikzcd}
A_{1}
\arrow[rr, "\chi"]
\arrow[d, "f_{1}"']
& & A_{1}
\arrow[d, "f_{1}"]
\\
B_{1}
\arrow[r, "d_{1}"']
& B_{0}
\arrow[r, "i_{0}"']
& B_{1}
\end{tikzcd}
\end{equation}
\item This diagram specifies a ``composition condition'': 
\begin{equation}
\label{eqn:isocomp}
\begin{tikzcd}[column sep = small]
& A_{1}
\arrow[ld, "\widehat{\chi}"']
\arrow[rd, "1"]
& \\
A_{2}
\arrow[rr, "d_{1}"']
& & A_{1}
\end{tikzcd}
\end{equation}
\end{itemize}
The pullback $\Lambda_{1} \times_{A_{0}} V_{1}$ is defined in 
\eqref{eqn:factorpullback},
and the morphism $\widehat{\chi}$ is defined using \eqref{eqn:isodom} 
via the universal property of the pullback:
\begin{equation*}
\begin{tikzcd}[row sep = small, column sep = small]
& A_{1}
\arrow[dd, dashed, "\widehat{\chi}"]
\arrow[ld, "{\langle d_{0},\, f_{1} \rangle}"']
\arrow[rddd, "\chi", bend left]
& \\
\Lambda_{1}
\arrow[dd, "\varphi_{1}"']
& & \\[-1em]
& A_{2}
\arrow[ld, "d_{0}"']
\arrow[rd, "d_{2}"]
\arrow[dd, phantom, "\lrcorner" rotate = -45, very near start]
& \\
A_{1}
\arrow[rd, "d_{1}"']
& & A_{1}
\arrow[ld, "d_{0}"]
\\
& A_{0} & 
\end{tikzcd}
\end{equation*}
\end{proposition}
\begin{proof}
By Definition~\ref{defn:strictfact} the pair $(\Lambda, V)$ 
is a strict factorisation system on $A$ if and only if 
the morphism 
$d_{1}(\varphi_{1} \times j_{1}) 
\colon \Lambda_{1} \times_{A_{0}} V_{1} \rightarrow A_{1}$ 
has an inverse. 

Given a morphism $\chi \colon A_{1} \rightarrow A_{1}$ satisfying axioms 
\eqref{eqn:isodom} and \eqref{eqn:isofibre} there is a morphism,
\[
	\big\langle \langle d_{0}, f_{1}\rangle, 
	\langle \chi, d_{1}f_{1} \rangle \big\rangle
	\colon A_{1} \longrightarrow \Lambda_{1} \times_{A_{0}} V_{1}
\]
which is the required inverse if axioms \eqref{eqn:isounique} 
and \eqref{eqn:isocomp} are satisfied. 
Conversely, an inverse to $d_{1}(\varphi_{1} \times j_{1}) 
\colon \Lambda_{1} \times_{A_{0}} V_{1} \rightarrow A_{1}$ 
exists exactly when there is 
a morphism $\chi$ which satisfies the axioms in the statement of 
Proposition~\ref{prop:isostructure}. 
The necessary diagram-chasing has been excluded. 
\end{proof}

The following example shows that the above characterisation yields the 
expected strict factorisation system on a small category $A$ given by a 
\emph{pre-opcartesian morphism} followed by a vertical morphism, and 
explains the meaning of the names given to the axioms above. 

\begin{example}[$\E = \Set$]
\label{ex:exchi}
A delta lens $(f, \varphi) \colon A \br B$ 
(see Example~\ref{ex:deltalens})
has a strict factorisation system $(\Lambda, V)$ if for all morphisms 
$w \colon a \rightarrow a'$ in $A$ there exists a 
unique morphism $\chi(w) \colon p(a, fw) \rightarrow a'$ such that 
$f\chi(w) = 1_{fa'}$ and $\chi(w) \circ \varphi(a, fw) = w$. 
\begin{equation*}
\begin{tikzcd}
& & p(a, fw)
\arrow[d, dashed, "{\chi(w)}"]
\\
A
\arrow[d, "f", shift left, harpoon]
& a
\arrow[d, phantom, "\vdots"]
\arrow[ru, "{\varphi(a, fw)}"]
\arrow[r, "w"']
& a' 
\arrow[d, phantom, "\vdots"]
\\
B
\arrow[u, "\varphi", shift left, harpoon]
& f a
\arrow[r, "fw"']
& fa'
\end{tikzcd}
\end{equation*}
The domain of $\chi(w)$ corresponds to \eqref{eqn:isodom} and the 
uniqueness to \eqref{eqn:isounique}.
The other two conditions on $\chi(w)$ correspond exactly to 
\eqref{eqn:isofibre} and \eqref{eqn:isocomp}, respectively.  
Thus $\chi$ provides a global way of characterising the chosen lifts 
$\varphi(a, fw)$ as \emph{pre-opcartesian morphisms}. 
\end{example}

The axioms of a delta lens imply that the chosen pre-opcartesian
morphisms compose to give a chosen pre-opcartesian morphism. 
This means that a delta lens with the strict factorisation system as 
given in Example~\ref{ex:exchi} is a split opfibration. 
Internally, there is also a strong similarity between 
Proposition~\ref{prop:internalsplitopf} and 
Proposition~\ref{prop:isostructure}, in terms of additional structure 
with axioms placed on an internal lens,
which motivates the following theorem. 

\begin{theorem}
\label{thm:johnstone}
An internal lens $(f, \varphi) \colon A \br B$ is an internal split
opfibration if and only the pair $(\Lambda, V)$ is a strict 
factorisation system on $A$.
\end{theorem}
\begin{proof}
Given an internal split opfibration $(f, \varphi) \colon A \br B$ 
equipped with a morphism 
$\psi \colon A_{1} \times_{A_{0}} B_{2} \rightarrow A_{1}$ 
as in Proposition~\ref{prop:internalsplitopf},
define a morphism $\chi' \colon A_{1} \rightarrow A_{1}$ as the 
following composite:
\begin{equation}
\label{eqn:chi}
\begin{tikzcd}
A_{1}
\arrow[r, "{\langle 1,\, i_{0}f_{1} \rangle}"]
& A_{1} \times_{B_{1}} B_{2}
\arrow[r, "\psi"]
& A_{1}
\end{tikzcd}
\end{equation}
It may be shown that $\chi'$ satisfies the axioms of 
Proposition~\ref{prop:isostructure} and thus yields a strict 
factorisation system $(\Lambda, V)$ on $A$. 

Conversely, given an internal lens $(f, \varphi) \colon A \br B$ 
equipped with  an endomorphism 
$\chi \colon A_{1} \rightarrow A_{1}$ as in 
Proposition~\ref{prop:isostructure}, define a morphism 
$\psi' \colon A_{1} \times_{B_{1}} B_{2} \rightarrow A_{1}$ 
as the following composite:
\begin{equation}
\label{eqn:psi}
\begin{tikzcd}
A_{1} \times_{B_{1}} B_{2}
\arrow[r, "{\langle \varphi_{1}p_{2} 
\langle d_{0} \times d_{0},\, \pi_{1} \rangle,\, \chi \pi_{0} \rangle}"]
&[+5em] A_{2}
\arrow[r, "d_{1}"]
& A_{1}
\end{tikzcd}
\end{equation}
It may be shown that $\psi'$ satisfies the axioms of 
Proposition~\ref{prop:internalsplitopf} and thus yields an internal 
split opfibration. 
\end{proof}

\begin{example}[$\E = \Set$]
Given a split opfibration as in Example~\ref{ex:exchi},
the function \eqref{eqn:psi} is defined as 
$\psi'(w, u, v) = \chi(w) \circ \varphi(p(a, u), v)$ as depicted below:
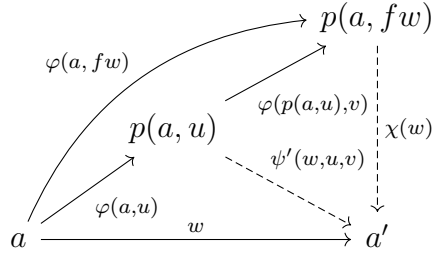
\begin{equation*}
\begin{tikzcd}
& & p(a, fw) 
\arrow[dd, dashed, "{\chi(w)}"]
\\
& p(a, u) 
\arrow[ru, "{\varphi(p(a, u), v)}"' pos = 0.25]
\arrow[rd, dashed, "{\psi'(w, u, v)}" pos = 0.3]
& \\
a
\arrow[ru, "{\varphi(a, u)}"']
\arrow[rr, "w"]
\arrow[rruu, bend left, "{\varphi(a,\, fw)}"]
& & a'
\end{tikzcd}
\end{equation*} 
\end{example}

Theorem~\ref{thm:johnstone} is important as it characterises 
internal split opfibrations as internal lenses 
$(f, \varphi) \colon A \br B$ where 
$A_{1} \cong \Lambda_{1} \times_{A_{1}} V_{1}$, giving the object 
of morphisms the universal property of the pullback 
\eqref{eqn:factorpullback}.
This mirrors Definition~\ref{defn:internaldopf} which states that 
internal discrete opfibrations are internal functors 
$f \colon A \rightarrow B$ where 
$A_{1} \cong \Lambda_{1} = A_{0} \times_{B_{0}} B_{1}$. 

\begin{corollary}
\label{cor:johnstone}
If $(f, \varphi) \colon A \br B$ is an internal split opfibration, 
then the following diagrams commute: 
\begin{equation*}
\begin{tikzcd}
\Lambda_{1} 
\arrow[d, "p_{1}"']
\arrow[r, "\varphi_{1}"]
& A_{1}
\arrow[d, "\chi"]
\\
A_{0}
\arrow[r, "i_{0}"']
& A_{1}
\end{tikzcd}
\quad
\begin{tikzcd}
A_{2}
\arrow[d, "d_{1}"']
\arrow[r, "{\langle \chi d_{1} (\chi \times \varphi_{1} 
\langle d_{0}, f_{1}\rangle),\, \chi d_{2} \rangle}"]
&[+5em] A_{2}
\arrow[d, "d_{1}"]
\\
A_{1}
\arrow[r, "\chi"']
& A_{1}
\end{tikzcd}
\quad
\begin{tikzcd}[column sep = small]
& A_{1}
\arrow[ld, "\chi"']
\arrow[rd, "d_{1}"]
& \\
A_{1}
\arrow[rr, "d_{1}"']
& & A_{0}
\end{tikzcd}
\quad 
\begin{tikzcd}[column sep = small]
& A_{1}
\arrow[ld, "\chi"']
\arrow[rd, "\chi"]
& \\
A_{1}
\arrow[rr, "\chi"']
& & A_{1}
\end{tikzcd}
\end{equation*}
\end{corollary}
\begin{proof}
These diagrams may be derived using the axioms in 
Proposition~\ref{prop:isostructure}. 
\end{proof}

\begin{remark}
The definition of an internal split opfibration in 
\cite[Ch.~2, Ex.~6]{Joh77} is equivalent to an internal 
lens $(f, \varphi) \colon A \br B$ equipped with a morphism 
$\chi \colon A_{1} \rightarrow A_{1}$ 
satisfying the diagrams in Proposition~\ref{prop:isostructure} and 
Corollary~\ref{cor:johnstone}. 
Therefore by Theorem~\ref{thm:johnstone}, the definition of internal 
split opfibrations provided in Definition~\ref{defn:internalsplitopf} 
agrees with the characterisation in \cite{Joh77} while requiring fewer 
axioms. 
\end{remark}

%----------------------------------------------------------------------%
% Section 6: Concluding remarks										   %
%----------------------------------------------------------------------%
\section{Concluding remarks} 
\label{sec:conclusion}

This paper has introduced several characterisations of internal 
split opfibrations in terms of both property and structure.
The core of this approach has been the treatment of internal split 
opfibrations as \emph{morphisms} between internal categories, 
given by an internal functor and an internal cofunctor which form 
an internal lens. 
This is unlike many previous approaches, which instead consider split 
opfibrations as \emph{objects} over a fixed base category. 
The treatment of internal split opfibrations as morphisms is not 
only of theoretical interest, but is important for the application 
of delta lenses in computer science where compositionality plays a 
central role. 

Another significant component of this treatment of internal split 
opfibrations is that it does not use any of the $2$-categorical 
features of $\CAT(\E)$. 
Instead the paper utilises two special classes of morphisms in 
$\CAT(\E)$ -- isomorphism-on-objects internal functors and internal
discrete opfibrations -- which satisfy the conditions of 
Lemma~\ref{lem:important}, together with a comonad on $\CAT(\E)$, 
given by the right d\'{e}calage construction. 
Consequently, the main characterisation of internal split 
opfibrations in Theorem~\ref{thm:main} may be further generalised 
by replacing $\CAT(\E)$ with a category $\mathcal{C}$ with pullbacks,
together with suitable classes of morphisms $(\mathcal{B}, \mathcal{D})$
and a comonad $D$, however further work is needed to identify 
interesting examples in this general framework.

The theory of internal lenses and internal split opfibrations 
developed in this paper clarifies and extends many results concerning
delta lenses in the computer science literature. 
For example, Proposition~\ref{prop:internallens} provides a simple
diagrammatic characterisation of the delta lens axioms in 
Example~\ref{ex:deltalens}, which further clarifies composition of 
delta lenses by Remark~\ref{rem:lensrepcomp}.
The characterisations of internal split opfibrations presented also 
provide the necessary and sufficient conditions for a delta lens to be
``least-change'', an important property for applications. 
Future work will apply these results to recent research on 
symmetric lenses \cite{JR17, JR19} with the goal of finding
similar conditions for universality. 

%-----------------------------------------------------------------------
% References - generated by BibTeX with \bibliographystyle{alpha}
%-----------------------------------------------------------------------

\end{document}